\documentclass{article}
\usepackage[utf8]{inputenc}
\usepackage{amsmath}
\usepackage{amssymb}
\usepackage{amsthm}
\usepackage{xcolor}
\usepackage{fullpage}
\usepackage{hyperref}
\usepackage{tikz}
\usepackage{soul}
\usepackage{mathrsfs}
\usepackage{enumitem}
\usepackage{subcaption}
\usetikzlibrary{arrows.meta}
    
\newtheorem{theorem}{Theorem}[section]

\newtheorem{remark}[theorem]{Remark}

\newtheorem*{theorem*}{Theorem}

\title{A Unified Geometric Perspective on Zygmund’s Maximal Function Conjecture and Related Questions for the Stein Hilbert Transform}
\author{Lingxiao Zhang  \thanks{The author is partially supported by the AMS-Simons travel grant.}}

\begin{document}
\maketitle

\begin{abstract}
The Zygmund vector-field maximal function conjecture is a long-standing open problem in harmonic analysis. It is notoriously difficult in part because the vector-field maximal function is non-translation-invariant and lies beyond the scope of finite-type theory. Its resolution would yield a Lebesgue differentiation theorem along vector fields, i.e., pointwise recovery of functions from averages along vector fields. We establish an \(L^2\) boundedness criterion that substantially enlarges the class of admissible geometric degeneracies. More precisely, the power-type decay condition in Bourgain’s theorem can be replaced by the much weaker logarithmic-polynomial decay \((\log\tfrac{1}{\tau})^{-p}\), \(p>2\). This condition no longer yields the power-type integrability underlying Bourgain’s argument, yet \(L^2\) boundedness persists, permitting degeneracies far beyond any power-law regime. Our result also provides a unified geometric perspective on finite-type, intermediate-degenerate, and highly non-finite-type vector fields, with their behavior organized through the angular-variation quantity \(w_x(t)\). This perspective places several classical boundedness criteria within a common hierarchy. Motivated by Lacey and Li’s work on the Stein vector-field Hilbert transform conjecture, we further introduce a non-centered rectangular maximal operator adapted to the variable geometry of vector fields and prove its weak-type \((1,1)\) boundedness. Together, these results provide new tools and a unified geometric framework for the study of two fundamental conjectural problems in non-translation-invariant analysis.
\end{abstract}

\section{Introduction}

The study of maximal and singular integral operators associated with lower-dimensional manifolds is a central theme in modern harmonic analysis, with connections to complex analysis, partial differential equations, and ergodic theory. In particular, Zygmund's conjecture addresses maximal functions associated with vector fields. 

The conjecture is notoriously difficult in part because the vector-field maximal function is non-translation-invariant and lies beyond the scope of classical finite-type theories. Its resolution would yield a Lebesgue differentiation theorem along vector fields---that is, the pointwise recovery of a function from its averages taken along the vector field. Classical differentiation theory is fundamentally tied to averaging over translation-invariant families, whereas vector-field averages introduce a spatially varying geometry which can be highly degenerate.

This paper focuses on advancing Zygmund's conjecture. 
Let $\Omega\subset \mathbb{R}^2$ be an open bounded set, and let $v$ be a $C^1$ vector field defined on a neighborhood of $\overline{\Omega}$. For a sufficiently small parameter $\epsilon_0>0$, we study
$$
M_vf(x) = \sup_{\epsilon< \epsilon_0} \left|A_\epsilon f(x)\right| = \sup_{\epsilon<\epsilon_0} \left|\frac{1}{\epsilon}\int_{-\epsilon}^\epsilon f(x+tv(x))\,dt \right| \cdot \chi_\Omega(x),
$$
where $\chi_\Omega$ denotes the characteristic function supported on $\Omega$.  
We establish a new boundedness criterion for $M_v$, relaxing the \emph{power-type decay} of Bourgain to a \emph{logarithmic polynomial decay}. This new condition no longer yields the power-type integrability underlying Bourgain's argument, yet $L^2$ boundedness persists, permitting degeneracies far beyond any power-law regime.

Since $\overline{\Omega}$ is compact and $v$ is $C^1$ on a neighborhood of $\overline{\Omega}$, we may fix a constant $B > 0$ such that $\|v\|_{C^1} \le B$ on this neighborhood. Throughout this paper, we choose the parameter $\epsilon_0 > 0$ sufficiently small depending only on $\Omega$ and $B$. This choice ensures both that the line segments $\left\{x + tv(x) : |t| \le \epsilon_0\right\}$ remain within the domain of $v$ for all $x \in \Omega$, and that $\epsilon_0$ satisfies the necessary smallness conditions.

For $x \in \Omega$ and $t \in [-\epsilon_0, \epsilon_0]$, we define the function
\begin{equation}
w_x(t) = \left|\det \left[ v(x+tv(x)), v(x)\right]\right|.
\end{equation}
Note that $w_x(t)$ equals the area of the parallelogram spanned by the vectors $v(x+tv(x))$ and $v(x)$. Informally, this quantity encodes both the relative angular deviation and the magnitude information of $v$ along the line segment $t\mapsto x+tv(x)$ emanating from $x$.

The following is Bourgain's theorem:
\begin{theorem}[\cite{MR1009171}]
If $v$ is $C^1$ and satisfies 
\begin{equation}\label{bourgain}
\bigg|\bigg\{ t\in [-\epsilon, \epsilon]: w_x(t) <\tau \sup_{[-\epsilon, \epsilon]} w_x(t)\bigg\}\bigg| \leq C\tau^{c_0}\epsilon, \quad \forall 0<\tau<1, \forall 0<\epsilon<\epsilon_0, \forall x\in \Omega,
\end{equation}
for some constants $0<c_0, C<\infty$. Then $M_v$ is bounded on $L^2$.
\end{theorem}

The following is the main result, which requires only a logarithmic polynomial decay:

\begin{theorem}\label{3}
If $v$ is $C^1$ and satisfies 
\begin{equation}\label{zhang}
\bigg|\bigg\{ t\in [-\epsilon, \epsilon]: w_x(t) <\tau \sup_{[-\epsilon, \epsilon]} w_x(t)\bigg\}\bigg| \leq C \left(\log \frac{1}{\tau}\right)^{-p}\epsilon,  \quad \forall 0<\tau<1, \forall 0<\epsilon<\epsilon_0, \forall x\in \Omega,
\end{equation}
for some constants $0<C<\infty$ and $2<p<\infty$. Then $M_v$ is bounded on $L^2$.
\end{theorem}

\begin{remark}
By a straightforward modification of the proof, the decay rate $\left(\log \frac{1}{\tau}\right)^{-p}$ in \eqref{zhang} can be replaced by iterated logarithmic variants, such as $\left(\log \frac{1}{\tau}\right) \left(\log \log \frac{1}{\tau}\right)^{-p}$ or $\left(\log \frac{1}{\tau}\right) \left(\log \log \frac{1}{\tau}\right) \left(\log \log \log \frac{1}{\tau}\right)^{-p}$. We do not pursue these routine refinements here.
\end{remark}

Unlike traditional frameworks that treat finite-type and non-finite-type operators separately, Section \ref{1.1} presents a unified geometric perspective on finite-type, intermediate-degenerate, and highly non-finite-type vector fields, situating several pivotal conditions from earlier foundational works within a single hierarchy. 
In Section \ref{refinement}, we prove Theorem \ref{3}. In Section \ref{sec3}, motivated by Lacey and Li's work on the Stein vector-field Hilbert transform conjecture, we introduce a non-centered rectangular maximal operator adapted to the variable geometry of vector fields and prove its weak-type $(1,1)$ boundedness. Although Theorem \ref{3} and Section \ref{sec3} study different operators, they are both driven by the same underlying vector-field geometry.

\subsection{A unified geometric view for boundedness criteria}\label{1.1}

The foundational insight of Stein and Wainger in the late 1970s explicitly introduced several central $L^p$ boundedness problems concerning the role of curvature. In this direction, Christ, Nagel, Stein, and Wainger \cite{CNSW} showed that $L^p$ boundedness of singular and maximal Radon transforms follows from a finite type condition: H\"ormander's condition. In particular, H\"ormander's condition is shown to be equivalent to a Jacobian nondegeneracy condition, which in turn implies the power-type integrability condition
\begin{equation}\label{J}
\int_B |J(\tau)|^{-\sigma}\,d\tau <\infty
\end{equation}
for sufficiently small $\sigma>0$, where $J(\tau)$ denotes an associated Jacobian.

Recall in Bourgain \cite{MR1009171}, the verification that any analytic $v$ satisfies condition \eqref{bourgain} is carried out via an equivalent condition: there exists $\sigma>0$ such that
\begin{equation}\label{J'}
\frac{1}{\epsilon}\int_{-\epsilon}^\epsilon \left(\frac{w_x(t)}{\sup_{[-\epsilon, \epsilon]}w_x(t)}\right)^{-\sigma}\,dt\lesssim 1, 
\end{equation}
uniformly in $x$ and $\epsilon$; see (2.3) in \cite{MR1009171}. This power-type integrability condition is of a similar form to \eqref{J}. The main difference is that no H\"ormander condition is assumed, so uniformity in $\epsilon$ is imposed explicitly across scales, rather than following automatically.

Finite type conditions, which involve curvature, have been extended to various contexts, including discrete operators, multi-parameter settings, and non-translation invariant cases. In the real analytic setting---where a certain finite-type condition holds automatically---Stein, Street, and Zhang \cite{ANALYTIC, ME} obtained a full characterization of bounded multi-parameter singular Radon transforms. In the non-analytic setting, such a finite-type condition does not automatically hold, yet it has long been understood that boundedness does not necessarily require finite-type conditions. For example, the operator
$$
f\mapsto \int_{-1}^1 f\left(x+t, y+e^{-\frac{1}{|t|}}\,\text{sgn}\,t\right)\,\frac{dt}{t}
$$
remains $L^p$-bounded despite the failure of finite type conditions (the curvature of the underlying curve at the origin vanishes); see Carbery, Christ, Vance, Wainger, and Watson \cite{CARBERYCHRIST}. 
For a result in the non-translation-invariant non-finite-type setting, see Zhang \cite{paper3}, which is of the same strength and form as \cite{CARBERYCHRIST}.
The approach in \cite{paper3} naturally degenerates to that of \cite{CARBERYCHRIST} when restricting to translation-invariant cases. Moreover, an integrability condition of the form \eqref{J'} crystallizes within its proof, as appears in Proposition 4.12 of \cite{paper3}.
Therefore power-type integrability conditions of the form \eqref{J'} also govern some non-finite-type settings.

In contrast, the present paper employs a weaker condition 
\eqref{zhang}, under which the integrability condition \eqref{J'} no longer holds. For the reader's convenience, we record here integrability conditions corresponding to different conditions in the literature to facilitate direct comparison. As mentioned above, Bourgain's condition \eqref{bourgain} is equivalent to the power-type integrability condition \eqref{J'}. For the Stein vector-field Hilbert transform conjecture, Lacey and Li (Remark 3.35 in \cite{MR2654385}) prove the $L^2$ boundedness of 
$$
H_vf(x) = \int_{-\epsilon_0}^{\epsilon_0} f(x+tv(x))\,\frac{dt}{t}
$$
under the condition
\begin{equation}\label{li} 
\bigg|\bigg\{ t\in [-\epsilon, \epsilon]: w_x(t) <\tau \sup_{[-\epsilon, \epsilon]} w_x(t)\bigg\}\bigg| \leq C e^{-\sigma(\log \frac{1}{\tau})^{c_1}} \epsilon,  \quad \forall 0<\tau<1, \forall 0<\epsilon<\epsilon_0, \forall x\in \Omega,
\end{equation}
for some constants $0<c_1, \sigma, C<\infty$. (This is of interest when $c_1<1$.)
This exponential-logarithmic decay remains the sharpest regularity condition known for Stein's conjecture.
Condition \eqref{li} is equivalent to the exponential-logarithmic integrability condition
\begin{equation}\label{2'}
\frac{1}{\epsilon}\int_{-\epsilon}^\epsilon \exp\!\left(\sigma\left(-\log \frac{w_x(t)}{\sup_{[-\epsilon, \epsilon]}w_x(t)}\right)^{c_1}\right) \,dt\lesssim 1,
\end{equation}
uniformly in $x$ and $\epsilon$, for some $\sigma>0$. Our condition \eqref{zhang} is equivalent to the logarithmic polynomial integrability condition  
\begin{equation}\label{3'}
\frac{1}{\epsilon}\int_{-\epsilon}^\epsilon \left(-\log \frac{w_x(t)}{\sup_{[-\epsilon, \epsilon]}w_x(t)}\right)^q \,dt \lesssim 1,
\end{equation}
uniformly in $x$ and $\epsilon$, for some $2<q<\infty$. In light of the relationships among \eqref{J'}, \eqref{2'}, and \eqref{3'}, the conditions introduced in this paper may be viewed as a natural continuation of a long-standing line of inquiry in the subject.

\begin{figure}[htbp]
\centering
\begin{tikzpicture}[xscale=0.9, yscale=0.6]
\draw[, fill=gray!10] (3.6, 0)--(-1.8, -1.8)--(-3.6,0)--(1.8,1.8)--cycle;
\draw[, fill=gray!25]  (1.6,0)--(-1.2,-1.2) -- (-1.6,0) --(1.2,1.2)-- cycle;
\draw[, fill=gray!45]    (0.4,0) -- (-0.6,-0.6) -- (-0.4,0) -- (0.6,0.6) -- cycle;
\draw [densely dotted] (-4.8,0) -- (4.8,0);
\draw [densely dotted] (-2.4,-2.4) -- (2.4,2.4);
\end{tikzpicture}
\caption{Metric geometries are generally defined by nested ``balls''. This figure displays a non-isotropic but homogeneous metric geometry.}
\label{fig:1}
\end{figure}
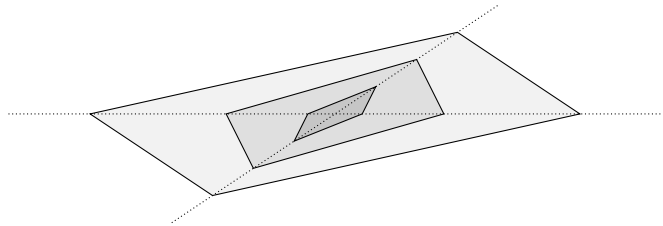

The hierarchy of decay conditions established above provides a glimpse into the underlying geometric structure, where $w_x(t)$ quantifies the angular variation of $v$. Crucially, \emph{less rigidity} is now imposed on this angular variation. For instance, the modulus of $w_x(t)$ can fluctuate between being of order $\sim 1$ and vanishing as rapidly as $e^{-1/|t|^\gamma}$ for $0 < \gamma < 1$.

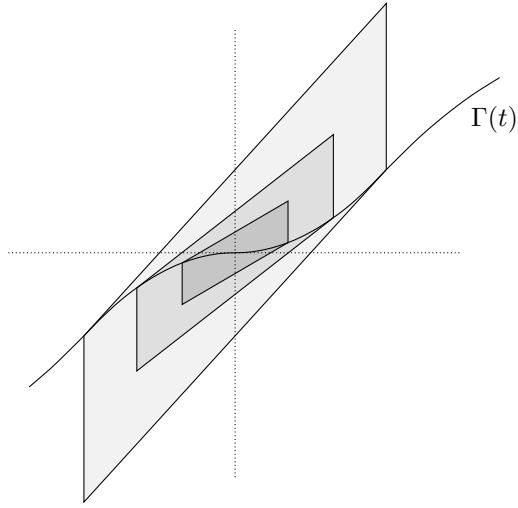
\begin{figure}[htbp]
\centering
\begin{tikzpicture}[xscale=1, yscale=1.1]
\draw [fill=gray!10] 
(-2,3)  -- (-4, 1)  -- (-6, -1) --(-6,-3) --(-4,-1) -- (-2, 1) --cycle;
\draw[fill=gray!25] (-2.7, 0.422) -- (-2.7, 1.422) -- (-4,0.5) -- (-5.3, -0.422) -- (-5.3, -1.422) -- (-4, -0.5) -- (-2.7, 0.422); 
\draw[fill=gray!45] (-3.3, 0.122) -- (-3.3, 0.622)  --(-4, 0.25) -- (-4.7, -0.122) --(-4.7, -0.622) -- cycle; 
\draw [densely dotted] (-4,-2.7)  -- (-4,2.7);
\draw [densely dotted] (-7,0) -- (-1,0);
\draw plot [smooth] coordinates {
(-0.5, 2.107) 
    (-0.75, 1.967) (-1.0, 1.81) 
    (-1.275, 1.615) (-1.55, 1.4) 
    (-1.775, 1.203) (-2,1) 
    (-2.2, 0.812) (-2.4, 0.64) 
    (-2.55, 0.524) (-2.7, 0.422) 
    (-2.85, 0.330) (-3, 0.25) 
    (-3.15, 0.180) (-3.3, 0.122) 
    (-3.5, 0.062) (-3.7, 0.022) 
    (-3.85, 0.005) (-4,0) (-4.15, -0.005) 
    (-4.3, -0.022) (-4.5, -0.062) (-4.7, -0.122) 
    (-4.85, -0.180) (-5, -0.25) 
    (-5.15, -0.330) (-5.3, -0.422) 
    (-5.45, -0.524) (-5.6, -0.64) 
    (-5.8, -0.812) (-6,-1) 
    (-6.225, -1.203) (-6.45, -1.4) 
       (-6.725, -1.615)
};
\draw (-1, 1.6) node[overlay, anchor=west]{$\Gamma(t)$};
\end{tikzpicture}
\caption{This figure displays the non-homogeneous metric geometry required to counteract the singularity near the origin of the operator associated with a representative curve $\Gamma(t)=\left(t, e^{-\frac{1}{|t|}}\,\text{sgn}\,t\right)$.}
\label{fig:2}
\end{figure}

To delve deeper into the underlying geometry, it is instructive to see its direct influence on operator boundedness and examine how it functions within the proof. For simplicity, consider the $\mathbb{R}^2$ case. With the Lebesgue measure, the role of a suitable metric geometry in securing boundedness is fundamental: in finite-type settings, a standard homogeneous---though not necessarily isotropic---metric suffices (Figure \ref{fig:1}). In non-finite-type settings, however, such as those in \cite{CARBERYCHRIST, paper3}, a dynamic ``rotation'' behavior is required to counteract the highly singular nature of the operator (Figure \ref{fig:2}). Both \cite{MR1009171} and the present paper implicitly incorporate this rotational behavior; specifically, $\sup_{[-\epsilon, \epsilon]} w_x(t)$ is derived from the smallest  rectangle containing all the parallelograms spanned by the fixed vector $\epsilon v(x)$ and the 
rotating vector $\epsilon v(x+t v(x))$ with $t$ varying in $[-\epsilon, \epsilon]$ (Figure \ref{fig:3}). 

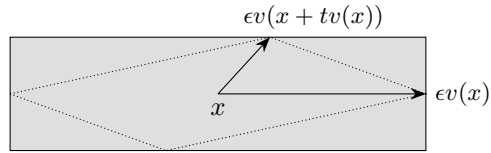
\begin{figure}[htbp]
\centering
\begin{tikzpicture}[xscale=0.55, yscale=1]
    \draw[fill=gray!25] (-5,-0.75) rectangle (5,0.75);
    
    \draw[-{Stealth[scale=1.2]}] (0,0) -- (5,0) node[overlay, anchor=west]{\small $\epsilon v(x)$};

    \draw [densely dotted] (-5,0) -- (1.25,0.75) -- (5,0) -- (-1.25,-0.75) -- cycle;
    
    \draw[-{Stealth[scale=1.2]}] (0,0) -- (1.25,0.75);
    \draw (2.3, 0.73) node[anchor=south]{\small $\epsilon v(x+tv(x))$};
    
    \draw (0,0) node[anchor=north]{$x$};
\end{tikzpicture}
\caption{This figure displays Bourgain's rectangle $R_{x,\epsilon}$ of center $x$, length $2\epsilon|v(x)|$, and width $2 \epsilon \sup_{[-\epsilon, \epsilon]} w_x(t)/|v(x)|$.}
\label{fig:3}
\end{figure}

Such underlying geometry is typically uncovered through distinct technical frameworks: the duality argument combined with Schur's test (\cite{MR1009171}), the $TT^\ast$ method (\cite{CNSW,ANALYTIC,paper3}), and time-frequency analysis (\cite{MR2654385}).

\section{Proof of the main result}\label{refinement}

In this section, we modify Bourgain’s arguments (\cite{MR1009171}) to prove Theorem \ref{3},. We retain his unified technical framework for maximal functions along vector fields while substantially weakening the decay condition. This extends the theory to encompass finite-type, intermediate-degenerate, and highly non-finite-type behaviors within a single framework, removing long-standing geometric limitations inherent in other approaches.

\begin{proof}[Proof of Theorem \ref{3}]
We need to prove the $L^2$ boundedness of $M_v$ under the logarithmic polynomial decay condition \eqref{zhang}. For comparison, we also prove simultaneously the boundedness of $M_v$ under exponential-logarithmic decay condition \eqref{li}.

Recall Bourgain \cite{MR1009171} used the power-type decay condition \eqref{bourgain}. The first instance in which condition \eqref{bourgain} is used in \cite{MR1009171} appears in equation (3.20), where it is shown that there exists $0<C<\infty$,
\begin{equation}\label{doubling}
\sup_{[-2\epsilon, 2\epsilon]} w_x(t) \leq C \sup_{[-\epsilon, \epsilon]} w_x(t), \quad \forall x\in \Omega, \forall 0<\epsilon<\epsilon_0.
\end{equation}
If we assume instead condition \eqref{zhang}, then since there exists $0<\tau_0<1$ (independent of $x$ and $\epsilon$) such that 
$$
C \left(\log \frac{1}{\tau_0}\right)^{-p} <1,
$$
we have
\begin{align*}
\left|\left\{t\in [-2\epsilon, 2\epsilon]: w_x(t) < \tau_0 \sup_{[-2\epsilon, 2\epsilon]} w_x(t)\right\}\right| < 2\epsilon, \quad \forall x\in \Omega, \forall 0<\epsilon<\epsilon_0.
\end{align*}
Hence 
\begin{align*}
\sup_{[-\epsilon, \epsilon]} w_x(t) \geq \tau_0 \sup_{[-2\epsilon, 2\epsilon]} w_x(t).
\end{align*}
Therefore \eqref{doubling} is obtained.

If we assume instead condition \eqref{li}, \eqref{doubling} follows similarly.

Lemma 3.21 of \cite{MR1009171} and its arguments remain the same under conditions \eqref{li} and \eqref{zhang}, respectively. 

The second place that requires modification is the last line of equation (3.27) in \cite{MR1009171}:
\begin{align*}
&\quad \int_{-\epsilon}^\epsilon \left( 1+\epsilon T |v(x_0)|^{-1}w_x(t)\right)^{-2}\,dt\\
&\leq \left|\left\{ t\in [-\epsilon, \epsilon]: w_x(t)<\tau \sup_{[-\epsilon, \epsilon]} w_x(t)\right\}\right| + \epsilon \left( \epsilon T |v(x_0)|^{-1} \tau \sup_{[-\epsilon, \epsilon]} w_x(t)\right)^{-2}\\
&\leq C\tau^{c_0} \epsilon + C\epsilon (\tau T\delta)^{-2}, \quad \forall 0<\tau<1.
\end{align*}
Recall in \cite{MR1009171}, Bourgain constructs the rectangle $R = R_{x,\epsilon}$ centered at $x$ with scale $\epsilon$, as a rectangle with length $2\epsilon|v(x)|$ along the direction $v(x)$ and with width $\delta = \delta(R_{x,\epsilon}):= \epsilon \sup_{[-\epsilon, \epsilon]} w_x(t)/|v(x)|$. The following estimates \eqref{2''}--\eqref{12} are implemented for $f$ whose Fourier transform is supported on $B(0,2T)\backslash B(0,T)$, and the parameter $T$ in \eqref{2''}--\eqref{12} satisfies $T\delta >1$.

If we assume instead condition \eqref{li}, we have 
\begin{equation}\label{2''}
\begin{aligned}
&\quad \int_{-\epsilon}^\epsilon \left( 1+\epsilon T |v(x_0)|^{-1}w_x(t)\right)^{-2}\,dt\\
&\leq \left|\left\{ t\in [-\epsilon, \epsilon]: w_x(t)<\tau \sup_{[-\epsilon, \epsilon]} w_x(t)\right\}\right| + \epsilon \left( \epsilon T |v(x_0)|^{-1} \tau \sup_{[-\epsilon, \epsilon]} w_x(t)\right)^{-2}\\
&\leq C e^{-\sigma \left(\log \frac{1}{\tau}\right)^{c_1}}  \epsilon + C\epsilon (\tau T\delta)^{-2}, \quad \forall 0<\tau<1. 
\end{aligned}
\end{equation}
If we assume instead condition \eqref{zhang}, we have 
\begin{equation}\label{3''}
\begin{aligned}
&\quad \int_{-\epsilon}^\epsilon \left( 1+\epsilon T |v(x_0)|^{-1}w_x(t)\right)^{-2}\,dt\\
&\leq \left|\left\{ t\in [-\epsilon, \epsilon]: w_x(t)<\tau \sup_{[-\epsilon, \epsilon]} w_x(t)\right\}\right| + \epsilon \left( \epsilon T |v(x_0)|^{-1} \tau \sup_{[-\epsilon, \epsilon]} w_x(t)\right)^{-2}\\
&\leq C\left(\log \frac{1}{\tau}\right)^{-p} \epsilon + C\epsilon (\tau T\delta)^{-2}, \quad \forall 0<\tau<1. 
\end{aligned}
\end{equation}

We therefore need corresponding modifications for (3.29) of \cite{MR1009171}:
\begin{align*}
\left\|A_\epsilon f\big|_R\right\|_2 &\lesssim \left( \frac{1}{T\delta}+\left(\frac{1}{\epsilon}\int_{-\epsilon}^\epsilon \left( 1+\epsilon T |v(x_0)|^{-1}w_x(t)\right)^{-2}\,dt\right)^{\frac{1}{2}}\right) \left\|f\left(\chi_{R'}\ast \psi_{\frac{1}{T}} \right)\right\|_2 \\
&\lesssim (T\delta)^{-c} \left\|f\left(\chi_{R'}\ast \psi_{\frac{1}{T}} \right)\right\|_2, \quad \text{for some } c>0,
\end{align*}
where $\psi$ is a fixed Schwarz function whose Fourier transform is supported in the unit disk, $\psi_{\frac{1}{T}} (x)= T^2 \psi(Tx)$, and $R'$ denotes the doubled rectangle of $R$. Recall condition \eqref{li} is of interest only when $c_1 <1$. By taking $0<\tau<1$ in \eqref{2''} to be $\tau_1$ such that
$$
e^{-\sigma \left(\log \frac{1}{\tau_1}\right)^{c_1}} = (\tau_1 T\delta)^{-2},
$$
we have
$$
\left(2\log \, (T\delta)\right)^{c_1} = \left(2\log \, \frac{1}{\tau_1} + \sigma \left(\log \frac{1}{\tau_1}\right)^{c_1}\right)^{c_1} \leq \left(2\log \, \frac{1}{\tau_1}\right)^{c_1} + \left(\sigma \left(\log \frac{1}{\tau_1}\right)^{c_1}\right)^{c_1} \lesssim \left(\log \, \frac{1}{\tau_1}\right)^{c_1},
$$
and thus 
\begin{equation}\label{11}
\begin{aligned}
\big\|A_\epsilon f\big|_R\big\|_2 &\lesssim \left( \frac{1}{T\delta} + e^{-\frac{\sigma}{2} \left(\log \frac{1}{\tau_1}\right)^{c_1}}\right) \left\|f\left(\chi_{R'}\ast \psi_{\frac{1}{T}} \right)\right\|_2 \\
&\lesssim e^{-\sigma' \left(\log T\delta \right)^{c_1}}\left\|f\left(\chi_{R'}\ast \psi_{\frac{1}{T}} \right)\right\|_2,
\end{aligned}
\end{equation}
for some $\sigma'>0$.

By taking $0<\tau<1$ in \eqref{3''} to be $\tau_2$ such that 
$$
\left(\log \frac{1}{\tau_2} \right)^{-p} = (\tau_2T\delta)^{-2},
$$
we have for any $\eta>0$,
$$
\tau_2^{2+\eta} =\frac{\tau_2^2}{\big(\frac{1}{\tau_2}\big)^\eta} \lesssim \frac{\tau_2^2}{\big(\log \frac{1}{\tau_2}\big)^p} = (T\delta)^{-2},
$$
and thus, by fixing a sufficiently small $\eta>0$,
\begin{equation}\label{12}
\begin{aligned}
\big\|A_\epsilon f\big|_R\big\|_2 &\lesssim \left( \frac{1}{T\delta} + \left(\log \frac{1}{\tau_2}\right)^{-\frac{p}{2}}\right) \left\|f\left(\chi_{R'}\ast \psi_{\frac{1}{T}} \right)\right\|_2 \\
&\lesssim \left( \log \left((CT\delta)^{\frac{2}{2+\eta}}\right)\right)^{-\frac{p}{2}} \left\|f\left(\chi_{R'}\ast \psi_{\frac{1}{T}} \right)\right\|_2 \\
&\lesssim \left(\log CT\delta \right)^{-\frac{p}{2}} \left\|f\left(\chi_{R'}\ast \psi_{\frac{1}{T}} \right)\right\|_2,
\end{aligned}
\end{equation}
for some $C>0$.

Section 4 in \cite{MR1009171} remains the same. 

Now let $f$ be an arbitrary $L^2$ function (no longer assume it is supported in an annulus of radius $\sim T$). Let $f= \sum_{T \,\text{dyadic}} f_T$ be a Littlewood-Paley decomposition with $\text{supp}\, \hat f_T \subset B(0,2T)\backslash B(0,T)$. In \cite{MR1009171},
$$
\Omega_{\epsilon,s}:= \left\{x\in \Omega : 2^{-s-1} \leq \epsilon |v(x)|^{-1} \sup_{[-\epsilon, \epsilon]} w_x(t) = \delta(R_{x,\epsilon}) <2^{-s}\right\},
$$
and
$$
\Omega_{\epsilon, s}':=\bigcup_{x\in \Omega_{\epsilon,s}} R_{x,\epsilon}',
$$
where $R_{x,\epsilon}'$ denotes the doubled rectangle of $R_{x,\epsilon}$.

The last place that requires modification is the last several equation displays in Page 20 in \cite{MR1009171}: cover $\Omega_{\epsilon,s}$ with rectangles $R_l$ of length $\sim \epsilon$ and width $\delta \sim 2^{-s}$ ($s\in \mathbb{N}$), then
$$
\left\|A_\epsilon\left(f_{2^{s+j}}\right)\chi_{\Omega_{\epsilon,s}} \right\|_2^2 \leq \sum_l \left\|A_\epsilon\left(f_{2^{s+j}}\right)\chi_{R_l}\right\|_2^2 \lesssim 2^{-2cj} \sum_l \left\|f_{2^{s+j}} \left(\chi_{R_l'}\ast \psi_{\frac{1}{2^{s+j}}}\right)\right\|_2^2,
$$
where $R_l'$ denotes a doubled rectangle of $R_l$, thus for equation (5.4) of \cite{MR1009171}:
$$
\|(5.4)\|_2 := \Bigg\|\sum_{j>0} \bigg[ \max_\epsilon \Big( \sum_{s>0} \left|A_\epsilon(f_{2^{s+j}})\right|^2 \chi_{\Omega_{\epsilon,s}} \Big)^{\frac{1}{2}} \bigg]\Bigg\|_2 \leq \sum_j \Bigg( \sum_{\epsilon, s} \bigg\| A_\epsilon(f_{2^{s+j}}) \chi_{\Omega_{\epsilon,s}} \bigg\|_2^2 \Bigg)^{\frac{1}{2}},
$$
they have
$$
\|(5.4)\|_2^2 \lesssim \bigg(\sum_j 2^{-cj}\bigg)\bigg(\sum_j 2^{-cj} \sum_{\epsilon, s} \left\|f_{2^{s+j}} \left(\chi_{\Omega_{\epsilon,s}'} \ast \psi_{2^{-s-j}}\right)\right\|_2^2 \bigg) \lesssim \sum_j 2^{-cj} \sum_s \|f_{2^{s+j}}\|_2^2 \lesssim \|f\|_2^2,
$$
where $0<\epsilon<\epsilon_0$ takes dyadic values.

With \eqref{11}, we have instead
$$
\left\|A_\epsilon\left(f_{2^{s+j}}\right)\chi_{\Omega_{\epsilon,s}} \right\|_2^2 \leq \sum_l \left\|A_\epsilon\left(f_{2^{s+j}}\right)\chi_{R_l}\right\|_2^2 \lesssim e^{-2\sigma'j^{c_1}} \sum_l \left\|f_{2^{s+j}} \left(\chi_{R_l'}\ast \psi_{\frac{1}{2^{s+j}}}\right)\right\|_2^2, \text{ for some } \sigma''>0,
$$
and thus for some $\sigma''>0$,
$$
\|(5.4)\|_2^2 \lesssim \bigg(\sum_j e^{-\sigma''j^{c_1}} \bigg) \bigg( \sum_j e^{-\sigma''j^{c_1}} \sum_{\epsilon, s} \left\|f_{2^{s+j}} \left(\chi_{\Omega_{\epsilon,s}'} \ast \psi_{2^{-s-j}}\right)\right\|_2^2 \bigg) \lesssim \sum_j e^{-\sigma''j^{c_1}} \sum_s \|f_{2^{s+j}}\|_2^2 \lesssim \|f\|_2^2.
$$

With \eqref{12}, we have instead: there exists a constant $j_0$ such that for $j\geq j_0$,
$$
\left\|A_\epsilon\left(f_{2^{s+j}}\right)\chi_{\Omega_{\epsilon,s}} \right\|_2^2 \leq \sum_l \left\|A_\epsilon\left(f_{2^{s+j}}\right)\chi_{R_l}\right\|_2^2 \lesssim \frac{1}{j^p} \sum_l \left\|f_{2^{s+j}} \left(\chi_{R_l'}\ast \psi_{\frac{1}{2^{s+j}}}\right)\right\|_2^2,
$$
and thus
$$
\|(5.4)\|_2^2 \lesssim \bigg(\sum_j \frac{1}{j^{\frac{p}{2}}}\bigg) \bigg(\sum_j \frac{1}{j^{\frac{p}{2}}} \sum_{\epsilon, s} \left\|f_{2^{s+j}} \left(\chi_{\Omega_{\epsilon,s}'} \ast \psi_{2^{-s-j}}\right)\right\|_2^2 \bigg) \lesssim \sum_j \frac{1}{j^{\frac{p}{2}}} \sum_s \|f_{2^{s+j}}\|_2^2 \lesssim \|f\|_2^2.
\qedhere
$$
\end{proof}

\section{A non-centered rectangular maximal operator}\label{sec3}

In this section, motivated by Lacey and Li's work, we introduce a non-centered rectangular maximal operator tailored to the underlying geometry and show its weak $(1,1)$ boundedness. This offers a novel tool for the subsequent study of the Zygmund and Stein conjectures.

Instead of the centered maximal function associated with rectangles $R_{x,\epsilon}$ as in \cite{MR1009171}, Lacey and Li \cite{MR2654385} introduced a non-centered maximal function. The non-centered maximal function is arguably more suitable, as operators associated with vector fields---particularly in the non-analytic setting---do not necessarily satisfy conditions such as \eqref{bourgain}, \eqref{li}, and \eqref{zhang}, meaning they lack inherent rigidity.

Let $v$ be a $C^1$ vector field on $\mathbb{R}^2$. Assume $|\nabla v|<B$. For any arbitrary rectangle $R= I\times J$ in the plane, assuming $|I|\geq |J|$, denote the length $L(R)=|I|$ and width $W(R) =|J|$. Lacey and Li consider the set of rectangles $R$ with $L(R)< \frac{1}{100B}$ and $W(R)<\frac{B}{100}$. Let
$$
V(R) = \left\{x\in R:  \left(\text{the angle between } v(x) \text{ and the long axis of } R\right) < W(R)/(2L(R))\right\}.
$$
Note $V(R)$ is an open subset of $R$. For any $0<w<\frac{B}{100}$ and $0<\delta<1$ fixed, Lacey and Li's maximal functions are defined as
$$
M_{v,\delta, w} f(x) = \sup_{\substack{|V(R)|\geq \delta |R|\\ w\leq W(R)<2w}} \frac{\chi_R(x)}{|R|} \int_R |f(y)|\,dy. 
$$
They proved that if $v$ is $C^{1+\eta}$ for some $\eta>0$, and if $M_{v,\delta, w}$ is bounded $L^p \to L^{p, \infty}$ for some $1<p<2$ with norm $\lesssim \delta^{-N}$ for all $w, \delta$, then the local Hilbert transform along $v$ is bounded on $L^2$.

However, such boundedness for $M_{v,\delta, w}$ has not been established. From our perspective, the geometry of a \emph{simply connected} rectangular domain may still \emph{not} be well adapted to the irregular behavior of operators associated with vector fields. In fact, we prove boundedness for the following maximal functions associated with $V(R)$:

\begin{theorem}\label{prop}
For any $0<\theta<\frac{1}{100}$ and $0<\delta<1$,
$$
\widetilde M_{v,\delta, \theta} f(x) = \sup_{\substack{|V(R)|\geq \delta |R|\\ \theta\leq W(R)/L(R)<2\theta}} \frac{\chi_{V(R)}(x)}{|V(R)|} \int_{V(R)} |f(y)|\,dy
$$    
is of weak type $(1,1)$, and hence bounded on $L^p$ for $1<p\leq\infty$; the operator norms are $\lesssim \delta^{-1}$.
\end{theorem}

Our goal is not to establish boundedness of the Hilbert transform associated with vector fields, but rather to isolate a structural shift at the level of non-centered rectangular maximal functions. Further extensions of this geometric framework---both for such maximal functions and for other components of the argument in \cite{MR2654385}---would likely be required to resolve Stein's conjecture.

\begin{proof}[Proof of Theorem \ref{prop}]
Fix arbitrary $f\in L^1(\mathbb{R}^2)$. For any $\lambda>0$, consider the open set $\left\{\widetilde M_{v, \delta, \theta}f>\lambda\right\}$. Fix an arbitrary compact subset $K$. $K$ is covered by finitely many $V(R_1), \ldots, V(R_N)$ such that $|V(R_i)|\geq \delta |R_i|$, $\theta \leq W(R_i)/L(R_i) <2\theta$, and
$$
\frac{1}{|V(R_i)|}\int_{V(R_i)} |f| >\lambda.
$$
Find a maximal disjoint subcollection of them, ordered by the lengths of the rectangles from largest to smallest, denoted by $V(R_1), \ldots, V(R_k)$. Let $R_1', \ldots, R_k'$ be the $10$-fold dilations of the rectangles $R_1, \ldots, R_k$, respectively, preserving their respective centers and eccentricities. Then 
$$
K\subseteq \bigcup_{i=1}^k R_i'.
$$
Indeed, if this were not the case, for a point $x_0\in K$ not covered by $R_1', \ldots, R_k'$, we have $x_0\in V(R_0)$ for some $R_0$. By how we choose the maximal disjoint collection, $V(R_0)$ must intersect $V(R_i)$ for some $1\leq i \leq k$ with $L(R_0)\leq L(R_i)$. Choose any point $z_0\in V(R_0) \cap V(R_i)$. The angle $\phi_1$ between $v(z_0)$ and the long axis of $R_0$ is less than $W(R_0)/(2L(R_0))$, and the angle $\phi_2$ between $v(z_0)$ and the long axis of $R_i$ is less than $W(R_i)/(2L(R_i))$.  Hence the angle $\phi_0$ between long axes of $R_0$ and $R_i$ satisfies
$$
\phi_0 \leq \phi_1 + \phi_2 < \frac{W(R_0)}{2L(R_0)} + \frac{W(R_i)}{2L(R_i)}<2\theta \leq 2\frac{W(R_i)}{L(R_i)}.
$$
Therefore for arbitrary $z\in R_0, y\in R_i$, the absolute value of the projection of $y-z$ along the short axis of $R_i$ is 
$$
\leq W(R_0)\cos \phi_0 + L(R_0) \sin \phi_0 \leq W(R_0)+ L(R_i) \phi_0 \leq 2W(R_i) + 2W(R_i) = 4W(R_i),
$$
and the absolute value of the projection of $y-z$ along the long axis of $R_i$ is
$$
\leq W(R_0)\sin \phi _0 + L(R_0) \cos \phi_0 \leq L(R_0) + L(R_0) =2L(R_0).
$$
Thus $R_0\subseteq R_i'$, and in particular, $x_0\in R_i'$; contradiction. Hence $K\subseteq \bigcup R_i'$. 

Therefore we have
\begin{align*}
|K| \leq \sum_{i=1}^k |R_i'| \leq 100 \sum_{i=1}^k |R_i| \leq \frac{100}{\delta} \sum_{i=1}^k |V(R_i)| \leq \frac{100}{\delta \lambda} \sum_{i=1}^k \int_{V(R_i)} |f| \leq \frac{100}{\delta \lambda} \|f\|_1.
\end{align*}
Hence $\widetilde M_{v, \delta, \theta}$ is of weak type $(1,1)$ with norm $\lesssim \delta^{-1}$.
\end{proof}

\bibliographystyle{amsalpha}
\bibliography{ref}

\vspace{2em}
\noindent
Lingxiao Zhang\\
Center for Applied Mathematics\\
Tianjin University\\
No. 92 Weijin Road,
Tianjin 300072,
P. R. China\\
Email: \href{mailto:lingxiaoz@tju.edu.cn}{lingxiaoz@tju.edu.cn}\\
MSC: Primary 42B25, 42B20

\end{document}